\documentclass[a4paper]{article}
\usepackage[T2A]{fontenc}
\usepackage[english]{babel}
\usepackage{mathtools}
\usepackage{amsmath}
\usepackage{amsfonts,amssymb}
\usepackage{amsthm}
\usepackage{esint}

\usepackage{mathrsfs}

\usepackage{graphicx}

\newcommand{\R}{\mathbb{R}}
\newcommand{\N}{\mathbb{N}}

\newcommand{\Z}{\mathbb{Z}}
\newcommand{\Ker}{{\operatorname {Ker}}}
\newcommand{\Img}{{\operatorname {range}}}
\newcommand{\Range}{{\operatorname {range}}}

\newcommand{\bx}{\mathbf{x}}
\newcommand{\bh}{\mathbf{h}}
\newcommand{\bg}{\mathbf{g}}
\newcommand{\bb}{\mathbf{b}}
\newcommand{\bff}{\mathbf{f}}
\newcommand{\dx}{\,\mathrm{d}x}

\newcommand{\ba}{\mathbf{a}}

\newcommand{\curl}{\operatorname{curl}}
\newcommand{\divv}{\operatorname{div}}

\newcommand{\bv}{\mathbf{v}}
 \newcommand{\by}{\mathbf{y}} 
 
 \newcommand{\bn}{\mathbf{n}}

\newtheorem{theorem}{Theorem}[section]
\newtheorem{lemma}{Lemma}[section]
\newtheorem{definition}{Definition}[section]
\newtheorem{remark*}{Remark}[]

\title{The Three-Dimensional Divergence-Curl Problem with Homogeneous Boundary Conditions and Pseudo-Harmonic Fields}
\author{A.V.\,Gorshkov}

\begin{document}
\maketitle
\abstract{This article addresses the solvability of the three-dimensional divergence-curl problem with a no-slip boundary condition. A solvability criterion is derived as an orthogonality condition of the vorticity function to pseudo-harmonic fields. A countable family of such fields, sufficient for the solvability of the three-dimensional problem in the exterior of a sphere, is also presented.\\ \\
{\it keywords: Div-curl problem, harmonic fields, no-slip condition}}

\section{Introduction}
This article studies the problem of finding a solenoidal vector field $\bv$ in a smooth domain $\Omega \subset \R^3$ with a prescribed curl function $\bff$
\begin{align}\label{mainsyst}
\begin{cases}
\curl \bv = \bff,\\
\divv \bv = 0,
\end{cases}
\end{align}
subject to the homogeneous boundary condition
\begin{align}\label{boundcond}
\bv|_{\partial \Omega} = 0.
\end{align}

This boundary value problem has applications in hydrodynamics, where the homogeneous Dirichlet condition --- known as the {\it no-slip condition} --- describes the interaction of a viscous medium with a solid body. The divergence-curl problem has been studied in various formulations with weaker boundary conditions. In \cite{KirchhartSchulz} it was studied without boundary conditions. The Cauchy problem for the three-dimensional divergence-curl system and $L_p$--$L_q$ estimates for the solution were obtained in \cite{GallayWayne}. Boundary value problems with prescribed tangential and normal boundary conditions were considered in \cite{Divcurl1}. The exterior Neumann problem was investigated in \cite{Divcurl2}, \cite{Divcurl3}. The magnetostatics problem with various boundary conditions was studied in variational form in \cite{BramblePasciak}. In the author's work \cite{G1}, the divergence-curl problem for incompressible flows with the no-slip condition was solved. The present paper extends that study to the three-dimensional case.

As is well known, solvability of the two-dimensional divergence-curl problem requires orthogonality of $\bff$ to harmonic functions \cite{G1}\cite{Quartapelle2}. In this article we extend this condition to higher dimensions by replacing harmonic fields in the orthogonality condition with {\it pseudo-harmonic} ones.

The vector Laplacian on $k$-forms admits the following invariant representation:
$$
\Delta = d^*d + dd^*.
$$
Here $d$ is the exterior derivative and $d^*$ is the codifferential. The kernel of this operator consists of {\it harmonic fields}.

For scalar functions and 2-forms, the two-dimensional Laplacian has only one non-zero term in each differential--codifferential pair:
$\Delta = d^*d = \divv(\nabla)$, or $\Delta = dd^* = \curl(\nabla^\perp)$.

We will prove solvability of the divergence-curl problem for {\it pseudo-harmonic fields}, in which only one pair of differentials --- namely $dd^*$ --- vanishes. For 2-forms on planar domains, these pseudo-harmonic fields coincide with harmonic functions, making the orthogonality principle with respect to pseudo-harmonic fields uniform across all dimensions.

On manifolds without boundary, by the adjointness of $d$ and $d^*$, the condition $dd^* = 0$ is equivalent to $d^* = 0$. However, for boundary value problems the situation changes significantly: the weaker requirement of orthogonality of $\bff$ to the kernel of $dd^*$ turns out to be not only necessary but also sufficient.

An important question is also the construction of a countable family of pseudo-harmonic fields, orthogonality to which is equivalent to the homogeneous Dirichlet boundary condition. For planar domains, a countable set of orthogonality conditions to harmonic functions takes the form ($\bx=(x_1,x_2)$):
\begin{itemize}
\item inside a disk of radius $r_0$:
\begin{align*}
\int_{|\bx|<r_0} f(\bx)(x_1+ix_2)^k \dx = 0,~k\in \N \cup \{0\},
\end{align*}

\item in the exterior of this disk:
\begin{align*}
\int_{|\bx|>r_0} \frac {f(\bx)}{(x_1+ix_2)^k} \dx = 0,~k\in \N \cup \{0\},
\end{align*}

\item in an annulus when $r_0 < \bx < r_1$:
\begin{align*}
\int_{r_0<|\bx|<r_1} \frac {f(\bx)}{(x_1+ix_2)^k} \dx = 0,~k\in \Z.
\end{align*}
\end{itemize}
In the general case, analogous conditions for two-dimensional domains can be obtained via the Riemann mapping (see \cite{G1} for details).

The article presents one example of a three-dimensional divergence-curl problem with a countable family of pseudo-harmonic fields $\{\ba_n\}_{n\in \N}$ satisfying
$$
d d^* \ba_n = \curl \curl \ba_n =0.
$$
The orthogonality condition $(\bff,\ba_n)_{L_2}=0$ will be sufficient for the solvability of the problem.

This is particularly relevant for the numerical solution of the divergence-curl boundary value problem with the homogeneous no-slip condition, where the constraint on the right-hand side is imposed not on an entire countable family but only on a finite subset thereof, making the boundary value of the solution approximately zero --- a condition known in hydrodynamics as the {\it partial-slip condition}.

\section{The Divergence-Curl Problem in Differential Forms}

The divergence-curl problem consists of finding a form $\omega$ on a Riemannian manifold $M$ satisfying the system
\begin{align} \label{divcurl}
\begin{cases}
d\omega = f,\\
d^*\omega = 0.
\end{cases}
\end{align}

Here $\omega$ is the unknown form in the space of smooth $k$-forms $\Omega^k(M)$, $d:\Omega^k(M) \to \Omega^{k+1}(M)$ is the exterior derivative, $d^*:\Omega^k(M) \to \Omega^{k-1}(M)$ is the codifferential defined by $d^*=(-1)^{n+nk+1} *d*$, $*:\Omega^k(M) \to \Omega^{n-k}(M)$ is the Hodge star operator.
The inner product on forms is defined by
$(\omega_1, \omega_2) = \int_M \omega_1 \wedge *\omega_2$.

The exterior derivative defines the cochain complex
$$
{\displaystyle \ldots \xrightarrow {} \Omega^{n-1}(M)\xrightarrow {d_{n}} \Omega^{n}(M)\xrightarrow {d _{n+1}} \Omega^{n+1}(M)\xrightarrow {} \ldots },
$$
while $d^*$ defines the chain complex
$$
{\displaystyle \ldots \xleftarrow {} \Omega^{n-1}(M)\xleftarrow {d^*_{n}} \Omega^{n}(M)\xleftarrow {d ^*_{n+1}} \Omega^{n+1}(M)\xleftarrow {} \ldots }.
$$

Since $d^2=0$, a necessary condition for solvability is $df=0$. On manifolds without boundary, the codifferential $d^*$ is adjoint to $d$:
\begin{align}\label{scalprod}
(d\omega_1, \omega_2) = (\omega_1, d^*\omega_2).
\end{align}
Consequently, $f$ must be orthogonal to all forms $\omega_1$ in the kernel of $d^*$:
\begin{align}\label{ort1}
(f,\omega_1) = (d\omega, \omega_1) = (\omega, d^*\omega_1) = 0~\forall \omega_1:~d^*\omega_1=0.
\end{align}

Suppose the geometry of $M$ is such that $d^*\omega=0$ implies the existence of a potential $a$ satisfying $\omega = d^*a$. Note that such a form $a$ may exist even for domains with non-trivial topology under appropriate boundary conditions.
Then from the divergence-curl system we obtain the equation
$$d d^* a= f,$$
whose solvability requires $f$ to be orthogonal to all forms in the kernel of $dd^*$:
\begin{align}\label{ort2}
( f,\omega_1) = 0~\forall \omega_1:~dd^*\omega_1=0.
\end{align}

\begin{definition} A form $\omega$ is called \emph{pseudo-harmonic} if $dd^*\omega=0$.
\end{definition}

On the other hand, $(dd^* \omega_1, \omega_2) = (\omega_1, dd^* \omega_2)$, and from (\ref{scalprod}) it follows that on manifolds without boundary the condition $dd^*\omega_1=0$ is equivalent to the condition (\ref{ort1}), so no additional orthogonality constraint arises.

For manifolds with boundary, however, condition (\ref{ort2}) becomes stronger than (\ref{ort1}). By Stokes' formula:
\begin{align*}
\int_M (dd^* a,  f) = 
\int_M ( a, dd^*  f) +
\int_{\partial M} (d^* a, * f)
- \int_{\partial M} (d^* f, *a).
\end{align*}

If we supplement the system (\ref{divcurl}) with the homogeneous boundary condition
$$
\omega\big|_{\partial M} = 0,
$$
then $d^* a =0$ on $\partial M$. Assuming the potential $a$ can be chosen so that $* a = 0$ on $\partial M$, orthogonality to pseudo-harmonic forms (\ref{ort2}) becomes a necessary condition for solvability.

\vskip0.5cm
We illustrate the orthogonality principle in concrete dimensions: a solenoidal two-dimen\-si\-onal vector field $\bv=(v_1,v_2)$ with homogeneous boundary conditions can be represented as the skew gradient
$$\bv = d^*\psi = \nabla^\perp \psi$$
of a stream function $\psi$, where $\nabla_\bx^\perp = (-\partial_{x_2}, \partial_{x_1})$.

The no-slip condition $\bv=0$ on $\partial \Omega$ implies boundary conditions on the stream function:
\begin{align}\label{streamboundarycond}
\psi(\bx)= \mathrm{const}, \quad \frac {\partial \psi (\bx)}{\partial \bn} =0,~\bx \in \partial \Omega.
\end{align}
Here and below $\bn$ denotes the outward unit normal to the boundary.

Multiplying the Poisson equation
$$
\Delta \psi = \curl \bv
$$
by a harmonic function $h$ and applying Green's formula gives
$$
\int _{\Omega } \bff h \,\dx = \int _{\Omega }h \,\Delta \psi \,\dx=\int _{\Omega }\psi\,\Delta h\,\dx+\int _{\partial \Omega }\left(h{\frac {\partial \psi}{\partial \bn }}-\psi{\frac {\partial h}{\partial \bn }}\right)\,dS = 0.
$$
Orthogonality of $\bff$ to harmonic functions is both necessary and sufficient for solvability of (\ref{mainsyst}), (\ref{boundcond}).

\vskip 0.5cm
For $n=2$ the complexes take the form:
\begin{align*}
&0 \xrightarrow {} \Omega^{0}(M)\xrightarrow {\nabla} \Omega^{1}(M)\xrightarrow {\curl} \Omega^{2}(M)\xrightarrow {} 0, \\
&0 \xleftarrow {} \Omega^{0}(M)\xleftarrow {-\divv} \Omega^{1}(M)\xleftarrow {\nabla^\perp} \Omega^{2}(M)\xleftarrow {} 0.
\end{align*}

The Laplacian $\Delta = d^*d + dd^*$ acting on scalar functions has only one non-zero pair:
$\Delta = d^*d = \divv(\nabla)$, or $\Delta = dd^* = \curl(\nabla^\perp)$, and $dd^*=0$ is equivalent to $\Delta=0$.
Therefore in two-dimensional divergence-curl problems, orthogonality to pseudo-harmonic fields reduces to orthogonality to harmonic functions, as described above.

\vskip 0.5cm
For $n=3$ the complexes are:
\begin{align*}
&0 \xrightarrow {} \Omega^{0}(M)\xrightarrow {\nabla} \Omega^{1}(M)\xrightarrow {\curl} \Omega^{2}(M)\xrightarrow {\divv} \Omega^{3}(M)\xrightarrow {} 0,\\
&0 \xrightarrow {} \Omega^{0}(M)\xleftarrow {-\divv} \Omega^{1}(M)\xleftarrow {\curl} \Omega^{2}(M)\xleftarrow {-\nabla} \Omega^{3}(M)\xleftarrow {} 0.
\end{align*}
The vector Laplacian contains both differential--codifferential pairs:
$$\Delta = d^*d + dd^* = -\nabla \divv + \curl^2.$$

If the flux through each connected component of the boundary is zero, then a solenoidal field $\bv$ has a vector potential, $\bv = \curl \ba$. Moreover, if $\bv \cdot \bn=0$, the potential $\ba$ can be chosen satisfying $\ba \times \bn = 0$. In this case $dd^* = \curl^2$ and condition (\ref{ort2}) becomes $\bff \perp \Ker \left ( \curl^2 \right )$.

\begin{definition} A 1-form (covector field) $\bv$ on $\R^3$  is called pseudo-harmonic if
$$dd^* \bv = \curl^2 \bv = 0.$$
\end{definition}

\begin{remark*}
In dimensions greater than two, forms lying in the kernels of both $d$ and $d^*$ also satisfy the orthogonality condition. On manifolds without boundary, harmonic fields lie in the kernels of both operators. However, orthogonality to harmonic fields is not sufficient for solvability of the boundary value problem (\ref{mainsyst}), (\ref{boundcond}). The broader class of pseudo-harmonic fields gives the correct set of forms in the necessary and sufficient solvability conditions (Theorem \ref{mainTh}).
\end{remark*}

\section{Solvability of the Divergence-Curl Problem}

Let $\Omega$ be a bounded smooth domain, $\mathcal D(\Omega)$ the space of smooth compactly supported functions in $\Omega$. Define
$$
\mathcal{V}=\{ \bv \in \mathcal D(\Omega),~\divv \bv=0 \},
$$
and let $\mathcal H$ be its closure in $L_2(\Omega)$. Since $\divv \bv \in L_2(\Omega)$, the normal trace $(\bv,\bn)\in H^{-1/2}(\partial \Omega)$ is well defined in $\mathcal H$.
We also introduce the spaces
\begin{align*}
&H(\curl,\Omega) = \{\ba\in L_2(\Omega),~\curl \ba\in L_2(\Omega)\},\\
&H(\curl^2,\Omega) = \{\ba \in L_2(\Omega),~\curl \ba\in L_2(\Omega),~\curl^2 \ba\in L_2(\Omega)\}.
\end{align*}
For $\bv \in H(\curl, \Omega)$, the trace $\bv \times \bn \in H^{-1/2}(\partial \Omega)$ is defined; for $\bv \in H(\curl^2, \Omega)$, so is $\curl \bv \times \bn \in H^{-1/2}(\partial \Omega)$. 

We then define
$$
H^0(\curl,\Omega) = \{\ba\in H(\curl,\Omega),~( \ba \times \bn)=0\}
$$
as the closure of $\mathcal D(\Omega)$ under the norm
$$
\|\ba \|_{H^0(\curl,\Omega)}^2 = \|\ba \|_{H^1(\Omega)}^2 + \| \curl \ba \|_{L_2(\Omega)}^2
$$
and
$$
H^0(\curl^2,\Omega) = \{\ba\in H(\curl^2,\Omega),~( \ba \times \bn)=0,~(\curl \ba \times \bn)=0\}
$$
as the closure of $\mathcal D(\Omega)$ under the norm
$$
\|\ba \|_{H^0(\curl^2,\Omega)}^2 = \|\ba \|_{H^1(\Omega)}^2 + \| \curl \ba \|_{L_2(\Omega)}^2+ \| \curl^2 \ba \|_{L_2(\Omega)}^2.
$$

Then $H^0(\curl^2,\Omega)$ is the set of functions $\ba$, satisfying $\ba \in H^0(\curl,\Omega)$, $\curl \ba \in H^0(\curl,\Omega)$.


Before stating the solvability theorem, we give a definition of a solution. The boundary condition (\ref{boundcond}) implies that on each connected component $\partial \Omega_i$ of $\partial \Omega$, the sufficient condition for the existence of a vector potential holds:
$$
\int_{\partial \Omega_i} (\bv, \bn) \dx = 0.
$$
By the Weyl decomposition \cite{Monk}\cite{Temam}, any $\bv \in L_2$ then has a vector potential $\ba$ with $\bv = \curl \ba$, and this potential can be chosen satisfying $\ba \times \bn=0$ (see \cite{BramblePasciak}).

The problem then reduces to:
\begin{align}
&\curl^2 \ba = \bff, \label{maineq1} \\
&\curl \ba  \times \bn |_{\partial \Omega} = 0,\label{mainsystcond1}\\
&\ba  \times \bn |_{\partial \Omega} = 0.\label{mainsystcond2}
\end{align}

Condition (\ref{mainsystcond1}) means orthogonality of the vector field $\bv$ to the boundary, and (\ref{mainsystcond2}) means tangentiality. Together they ensure (\ref{boundcond}).

\begin{definition} A \emph{solution} $\bv$ of problem (\ref{mainsyst}), (\ref{boundcond}) is the curl of a function $\ba\in H(\curl^2, \Omega)$ satisfying (\ref{maineq1})--(\ref{mainsystcond2}).
\end{definition}

The problem (\ref{maineq1})--(\ref{mainsystcond2}) is also overdetermined, as was the case in two dimensions for the Poisson equation with boundary conditions (\ref{streamboundarycond}). In this section we prove the theorem on unique solvability, which we precede with a theorem and a lemma.

\begin{theorem}[Closed Range Theorem] \label{thm1} 
Let $X$ and $Y$ be Hilbert spaces, and let $T: D_T \subset X\to Y$ be a closed operator with a densely defined domain $D_T$ that is compactly embedded in $X$. Then $\Img~T$ is closed in $Y$. 
\end{theorem}

Here, a closed operator is understood as an operator with a closed graph. The proof of this theorem can be found in \cite{Arnold}. Next, we prove a lemma on the closedness of the range of $\curl^2$:

\begin{lemma}
Let $\Omega\subset\R^3$ be a bounded smooth domain. The operator $\curl^2:H^0(\curl^2,\Omega)\subset L_2(\Omega)\to L_2(\Omega)$ is closed, and its range is closed.
\end{lemma}

\begin{proof}
Let $C$ be the operator $\curl$ in $L_2(\Omega)$ with the domain $D_C=H_0(\curl,\Omega)$. It is closed and densely defined, and $C^*$ is $\curl$ on $H(\curl,\Omega)$; in particular $C\subset C^*$.
The operator $\curl^2=C^2$ is defined on $D_{C^2}=H^0(\curl^2,\Omega)$. Denote $R=\Range C$, $X=(\Ker C)^\perp$.

First we prove that $R$ is closed. Consider operator $T=C|_{D_T}$ on $D_T=D_C \cap (\Ker C)^\perp$. Since $\Ker C\subset D_C$, the orthogonal projection onto $(\Ker C)^\perp$ maps $D_C$ onto $D_T$ and $C\ba=CP\ba$; hence $D_T$ is dense in $(\Ker C)^\perp$ and $\Range T=R$. The graph of $T$ equals $G(C)\cap \left( (\Ker C)^\perp \times L_2 \right )$, so $T$ is closed. Since $\nabla H^1_0(\Omega)\subset\Ker C$, every $\bb\in D_T$ satisfies $\divv\bb=0$. By the Gaffney inequality\cite{Picard},
$$\|\bb\|_{H^1(\Omega)}\le c\,(\|\bb\|_{L_2}+\|\curl\bb\|_{L_2}),\qquad \bb\in D_T.$$
By Rellich’s theorem $D_T$ norm is compactly embedded in $(\Ker C)^\perp$. Theorem~\ref{thm1} gives that $R$ is closed. Moreover, $T$ is injective and $T^{-1}:R\to (\Ker C)^\perp$ is closed and defined on the whole Hilbert space $R$, hence bounded:
\begin{align} \label{auto:1}
\|\bb\|\le c\,\|\curl\bb\|,\qquad \bb\in D_T. 
\end{align}

Now we will act by $C$ on $R$. Let $W=R\cap D_C$. Since $C\subset C^*$, we have $\Ker C\subset\Ker C^*=R^\perp$, hence $R\subset(\Ker C)^\perp$ and
\begin{align} \label{auto:2}
W\subset D_T. 
\end{align}
Consequently (\ref{auto:1}) holds on $W$ and $\divv\bb=0$ for $\bb\in W$.

Next we will find the range of $\curl^2$.  If $\ba\in D_{C^2}$, then $C\ba\in R\cap D_C=W$. Conversely, if $\bb\in W$, then $\bb=C\ba$ for some $\ba\in D_C$, and $C\ba=\bb\in D_C$ means $\ba\in D_{C^2}$. Thus $C(D_{C^2})=W$ and $\Range C^2=C(W)$.

The operator $\curl^2$ is closed. Indeed, let $\ba_n\in D_{C^2}$, $\ba_n\to\ba$, $\curl^2\ba_n\to\bg$ in $L_2$. Put $\bb_n=C\ba_n\in W$. By (\ref{auto:1}) and (\ref{auto:2}),
$$\|\bb_n-\bb_m\|\le c\,\|C\bb_n-C\bb_m\|=c\,\|\curl^2\ba_n-\curl^2\ba_m\|\to0,$$
so $\bb_n\to\bb$ in $L_2$. Since $C$ is closed and $\ba_n\to\ba$, $C\ba_n\to\bb$, we get $\ba\in D_C$ and $C\ba=\bb$. Since $\bb_n\to\bb$ and $C\bb_n\to\bg$, we get $\bb\in D_C$ and $C\bb=\bg$. Hence $\ba\in D_{C^2}$ and $\curl^2\ba=\bg$.

Now we are ready to prove, that $\Range\curl^2$ is closed. Let $X'$ be a closure of $W$ in $L_2$ and $T'=C|_W:X'\to L_2$. The graph of $T'$ is $G(C)\cap(R\times L_2)$, hence closed, because $C$ is closed and $R$ is closed. $W$ is dense in $X'$ by definition. By (\ref{auto:2}) and the Gaffney inequality, the graph norm of $T'$ controls the $H^1$-norm on $W$, so by Rellich’s theorem $W$ is compactly embedded in $X'$. Theorem~\ref{thm1} shows that $\Range T'=C(W)$ is closed. But $C(W)$ coincides  with $\Range C^2$.
\end{proof}

\begin{theorem}\label{mainTh}
Let $\Omega\subset\R^3$ be a bounded smooth domain and $\bff\in L_2(\Omega)^3$. The problem \eqref{mainsyst}, \eqref{boundcond} has a solution if and only if
$$\int_\Omega(\bff,\bh)\,d\bx=0\quad\text{for all }\bh\in H(\curl^2,\Omega)\text{ with }\curl^2\bh=0.$$
The field $\bv$ of this solution is unique.
\end{theorem}

\begin{proof}
Necessity of the orthogonality condition follows from the definition of the solution and the generalized Stokes formula. Multiplying (\ref{maineq1}) by $\bh$ with $\curl^2 \bh =0$ and integrating by parts, all boundary terms vanish due to the boundary conditions, giving
\begin{align*}
\int_\Omega (\curl^2 \ba, \bh) \dx = 
\int_\Omega (\ba, \curl^2 \bh) \dx -
\int_{\partial \Omega} (\curl \ba \times \bn, \bh)\,ds - 
\int_{\partial \Omega} (\ba \times \bn, \curl \bh)\,ds = 0
\end{align*}
for all $\bh \in \Ker \left ( \curl^2 \right )$.

%
%

Let us prove the sufficiency. In view of the boundary conditions (\ref{mainsystcond1}) and (\ref{mainsystcond2}), the operator $\curl^2$
$$\curl^2: L_2(\Omega) \to L_2(\Omega)$$ 
with the domain $H^0(\curl^2,\Omega)$ has an adjoint operator $\curl^{2,*}$
$$\curl^{2,*}: L_2(\Omega) \to L_2(\Omega),$$
defined on $H(\curl^2,\Omega)$. By Banach's closed range theorem, the range of $\curl^{2,*}$ on $H(\curl^2,\Omega)$ is also closed and  holds that: 
$$\Img \curl^2 = \left ( \Ker \curl^{2,*} \right )^\perp.$$ 
Thus, for any $\bff \in L_2$ orthogonal to the kernel of $\curl^{2,*}$, there exists a solution $\ba \in H^0(\curl^2,\Omega)$ to the problem (\ref{maineq1})–(\ref{mainsystcond2}).

Uniqueness follows from Stokes' formula. Let $\ba$ be a solution with $\bff=0$. Testing against $\ba$:
\begin{align*}
\int_\Omega (\curl^2 \ba,\ba) \dx= 
\int_\Omega (\curl \ba, \curl \ba) \dx + 
\int_{\partial \Omega} (\curl \ba \times \bn, \ba) \dx =
\int_\Omega (\curl \ba, \curl \ba) \dx = \\ \|\curl \ba\|^2_{L_2} = 0.
\end{align*}
Hence $\bv = \curl \ba \equiv 0$.  
\end{proof}

\begin{remark*}
The solenoidality of $\bff$ need not be imposed separately in Theorem~\ref{mainTh}: it is already contained in the orthogonality condition. Indeed, for every $\varphi\in H^1(\Omega)$ the gradient $\nabla\varphi$ belongs to $H(\curl^2,\Omega)$ and satisfies $\curl^2\nabla\varphi=0$, so

$$\int_\Omega(\bff,\nabla\varphi)\,d\bx=0\qquad\text{for all }\varphi\in H^1(\Omega).$$

Taking $\varphi\in\mathcal D(\Omega)$ gives $\divv\bff=0$ in $\Omega$, and then, by the Gauss--Ostrogradsky formula, taking arbitrary $\varphi\in H^1(\Omega)$ gives $\bff\cdot\bn=0$ on $\partial\Omega$. The boundary condition (\ref{boundcond})  implies $\bff\cdot\bn=0$. Indeed, if $(\bff,\bn) \neq 0$ on $\partial \Omega$, then the tangential part of $\bv$ on the boundary is non-zero. Nevertheless this condition does not appear explicitly in Theorem \ref{mainTh}
\end{remark*}


\section{The Divergence-Curl Problem in Sphe\-ri\-cal Do-\\mains}

We now apply the general theory to a three-dimensional divergence-curl problem in the exterior of a sphere. 
Although Theorem~\ref{mainTh} concerns bounded domains, we present an example in which $\Omega$ is the exterior of a sphere. There, a countable family of fields from $\Ker\curl^2$ is constructed explicitly, without appealing to the theorem.

We construct a countable family of pseudo-harmonic fields $\ba_n \in \Ker(\curl^2)$ that determine its solvability. These fields will also be solenoidal, hence harmonic ($\Delta \ba_n =0$), while $\curl \ba_n$ remains non-zero.

In the domain $B_{r_0}=\left \{\bx \in \R^3 ~\big | ~ |\bx|>r_0 \right \}$ for some fixed $r_0>0$, consider the exterior divergence-curl problem (\ref{mainsyst}) with no-slip boundary conditions on the sphere and at infinity:
\begin{align}
&\bv \big |_{|\bx|=r_0} = 0, \label{boundcondsphere} \\
&\bv \to 0,~|\bx|\to \infty. \label{boundcondinfsphere}
\end{align}

We seek the solution as an expansion in vector spherical harmonics:
\begin{align*}
\bv =\sum _{\ell =0}^{\infty }\sum _{m=-\ell }^{\ell }\left(V_{\ell m}^{r}(r)\mathbf {Y} _{\ell m}+V_{\ell m}^{(1)}(r)\mathbf {\Psi}_{\ell m}+V_{\ell m}^{(2)}(r)\mathbf {\Phi}_{\ell m}\right).
\end{align*}

Here the vector spherical harmonics are defined by
\begin{align*}
&\mathbf {Y} _{\ell m}=Y_{\ell }^{m}{\hat {\mathbf {r} }},\\
&\mathbf {\Psi } _{\ell m}=r\nabla Y_{\ell }^{m},\\
&\mathbf {\Phi } _{\ell m}=\mathbf {r} \times \nabla Y_{\ell }^{m}
\end{align*}
where $Y_{\ell }^{m}$ is the scalar spherical harmonic
$$
Y_{\ell }^{m}(\theta ,\varphi )={\sqrt {{\frac {(2\ell +1)}{4\pi }}{\frac {(\ell -m)!}{(\ell +m)!}}}}\,P_{\ell }^{m}(\cos {\theta })\,e^{im\varphi }.
$$
and $P_{\ell }^{m}$ is the associated Legendre function
$$P_{\ell }^{m}(x)={\frac {(-1)^{m}}{2^{\ell }\ell !}}(1-x^{2})^{m/2}\ {\frac {d^{\ell +m}}{dx^{\ell +m}}}(x^{2}-1)^{\ell },~
-\ell \leq m \leq \ell.
$$

The divergence and curl are expressed as follows:
\begin{align*}
&\nabla \cdot \bv =\sum _{\ell =0}^{\infty }\sum _{m=-\ell }^{\ell }\left({\frac {dV_{\ell m}^{r}}{dr}}+{\frac {2}{r}}V_{\ell m}^{r}-{\frac {\ell (\ell +1)}{r}}V_{\ell m}^{(1)}\right)\mathbf Y_{\ell m},
\\
&\nabla \times \bv = \sum _{\ell =0}^{\infty }\sum _{m=-\ell }^{\ell }\Bigg(-{\frac {\ell (\ell +1)}{r}}V_{\ell m}^{(2)}\mathbf {Y} _{\ell m}-\left({\frac {dV_{\ell m}^{(2)}}{dr}}+{\frac {1}{r}}V_{\ell m}^{(2)}\right)\mathbf {\Psi } _{\ell m}+\\
&\left(-{\frac {1}{r}}V_{\ell m}^{r}+ {\frac {dV_{\ell m}^{(1)}}{dr}}+{\frac {1}{r}}V_{\ell m}^{(1)}\right)\mathbf {\Phi } _{\ell m}\Bigg).
\end{align*}

Denoting the spherical components of $\bff$ by $(f_{\ell m}^r, f_{\ell m}^{(1)}, f_{\ell m}^{(2)})$, the system (\ref{mainsyst}) in terms of harmonic coefficients reads:
\begin{align}
&-{\frac {\ell (\ell +1)}{r}}V_{\ell m}^{(2)} = f_{\ell m}^r, \label{mainsyst_sphere_eq1}\\
&-{\frac {dV_{\ell m}^{(2)}}{dr}}-{\frac {1}{r}}V_{\ell m}^{(2)} = f_{\ell m}^{(1)}, \label{mainsyst_sphere_eq2}\\
&-{\frac {1}{r}}V_{\ell m}^{r}+ {\frac {dV_{\ell m}^{(1)}}{dr}}+{\frac {1}{r}}V_{\ell m}^{(1)} = f_{\ell m}^{(2)},\label{mainsyst_sphere_eq3}\\
&\frac {dV_{\ell m}^{r}}{dr}+{\frac {2}{r}}V_{\ell m}^{r}-{\frac {\ell (\ell +1)}{r}}V_{\ell m}^{(1)} = 0.\label{mainsyst_sphere_eq4}
\end{align}

From the first equation we obtain $f_{\ell m}^r(r_0) = 0$, i.e., $(\bff,\bn)=0$ on the boundary --- a condition that holds for arbitrary domains (see the remark in the previous section).

For spherical domains, an additional boundary condition on a second component of $\bff$ can be derived. The solenoidality condition gives:
\begin{align}\label{divf}
\frac {df_{\ell m}^{r}}{dr}+{\frac {2}{r}}f_{\ell m}^{r}-{\frac {\ell (\ell +1)}{r}}f_{\ell m}^{(1)} = 0,
\end{align}
from which a further boundary condition follows:
\begin{align*}
r_0 \frac {df_{\ell m}^r(r_0)}{dr} \Big |_{r=r_0} - \ell(\ell+1)f_{\ell m}^{(1)}(r_0) = 0.
\end{align*}

The system (\ref{mainsyst_sphere_eq1})--(\ref{mainsyst_sphere_eq4}), supplemented by solenoidality (\ref{divf}), is degenerate: equation (\ref{mainsyst_sphere_eq2}) follows from (\ref{mainsyst_sphere_eq1}) and (\ref{divf}), while the last two equations (\ref{mainsyst_sphere_eq3}), (\ref{mainsyst_sphere_eq4}) yield integral constraints on $f_{\ell m}^{(2)}$. Since $\mathbf{\Psi}_{\ell m} = \mathbf{\Phi}_{\ell m} = 0$ for $\ell=0$, we proceed with $\ell=1,\ldots,\infty$, $m\in\Z \cap [-\ell,\ell]$.

The solution basis of the system
\begin{align*}
&\frac {dV_{\ell m}^{r}}{dr}+{\frac {2}{r}}V_{\ell m}^{r}-{\frac {\ell (\ell +1)}{r}}V_{\ell m}^{(1)} = 0, \\
&-{\frac {1}{r}}V_{\ell m}^{r}+ {\frac {dV_{\ell m}^{(1)}}{dr}}+{\frac {1}{r}}V_{\ell m}^{(1)} = f_{\ell m}^{(2)}
\end{align*}
consists of two vector functions
\begin{align*}
\begin{bmatrix}-\frac{r^{-2-\ell}}{\ell+1} \\r^{-2-\ell}\\\end{bmatrix},
\quad
\begin{bmatrix}\frac{r^{\ell-1}}\ell \\r^{\ell-1}\\\end{bmatrix}.
\end{align*}

The system can be solved explicitly with the boundary condition at infinity (\ref{boundcondinfsphere}):
\begin{align}
&V_{\ell m}^{r} = \frac{\ell (\ell +1)}{r^{2+\ell}}\int_{r_0}^r
s^{2+\ell}  f_{\ell m}^{(2)}(s)\,ds +
\frac{\ell (\ell +1)}{r^{1-\ell}}\int_{r}^\infty
s^{1-\ell}  f_{\ell m}^{(2)}(s)\,ds,  \label{sol_sphere_1}\\
&V_{\ell m}^{(1)} = - \frac{\ell}{r^{2+\ell}}\int_{r_0}^r
s^{2+\ell}  f_{\ell m}^{(2)}(s)\,ds +
\frac{(\ell +1)}{r^{1-\ell}}\int_{r}^\infty
s^{1-\ell}  f_{\ell m}^{(2)}(s)\,ds. \label{sol_sphere_2}
\end{align}

The integral solvability condition in terms of spherical coefficients is:
\begin{align}\label{ortcondspherepolar}
\int_{r_0}^\infty
s^{1-\ell}  f_{\ell m}^{(2)}(s)\,ds = 0.
\end{align}

Returning to the original vector function $\bff$ via
\begin{align*}
f_{\ell m}^{(2)}(r) = \frac1{\ell (\ell +1)} \iint_{|\bx|=1} \bff(r\bx) \cdot \overline{\mathbf {\Phi}_{\ell m}}(\bx)\,dS,
\end{align*}
we obtain, for $\ell=1,\ldots,\infty$, $m\in\Z \cap [-\ell,\ell]$:
\begin{align}\label{ortcondsphere}
\iiint_{B_{r_0}} \frac{\bff(\bx) \cdot \overline{\mathbf {\Phi}_{\ell m}}\left (\bx/|\bx| \right )}
{|\bx|^{1 + \ell}} \dx =0.
\end{align}

Thus, solvability requires $\bff(\bx)$ to be orthogonal to the family $\{\mathcal H_{\ell m}\}$ of functions $$\mathcal H_{\ell m} = \frac {\mathbf{\Phi}_{\ell m} } {r^{\ell+1}}.$$ We verify that these fields are indeed pseudo-harmonic, i.e.,
$$\curl^2 \mathcal H_{\ell m} =0.$$

We use the curl identities for vector spherical harmonics, valid for any differentiable function $g(r)$:
\begin{align*}
\nabla \times \left(g(r)\mathbf {Y} _{\ell m}\right)&=-{\frac {1}{r}}g \mathbf {\Phi } _{\ell m},\\
\nabla \times \left(g(r) \mathbf {\Psi } _{\ell m}\right)&=\left({\frac {dg}{dr}}+{\frac {1}{r}}g\right)\mathbf {\Phi } _{\ell m},\\
\nabla \times \left(g(r)\mathbf {\Phi } _{\ell m}\right)&=-{\frac {\ell (\ell +1)}{r}}g\mathbf {Y} _{\ell m}-\left({\frac {dg}{dr}}+{\frac {1}{r}}g\right)\mathbf {\Psi } _{\ell m},
\end{align*}
Taking $g=r^{-\ell-1}$ and $g=r^{-\ell-2}$, we obtain
\begin{align*}
&\nabla \times \frac{\mathbf \Phi_{\ell m}}{r^{\ell+1}} = - \frac{\ell(\ell+1)}{r^{\ell+2}}\mathbf Y_{\ell m} + \frac \ell{r^{\ell+2}} \mathbf \Psi_{\ell m},\\
&\nabla \times \frac{\mathbf Y_{\ell m}}{r^{\ell+2}} = -\frac 1{r^{\ell+3}} \mathbf \Phi_{\ell m},\\
&\nabla \times \frac{\mathbf \Psi_{\ell m}}{r^{\ell+2}} =
-\frac {\ell+1}{r^{\ell+3}} \mathbf \Phi_{\ell m}.
\end{align*}
Applying the curl once more, we find
$$
\curl^2 \mathcal H_{\ell m} = \nabla \times \nabla \times \mathcal H_{\ell m} = -\frac{\ell(\ell+1) - \ell(\ell+1)}{r^{\ell+3}}\mathbf \Phi_{\ell m}  = 0.
$$

\begin{remark*}
These fields are also divergence-free. Indeed, $\mathbf\Phi_{\ell m}=\mathbf r\times\nabla Y_\ell^m$ has no radial component and $\nabla\cdot\mathbf\Phi_{\ell m}=0$, so for any function $g(r)$
\begin{align*}
\nabla \cdot \left(g(r){\mathbf\Phi } _{\ell m}\right)&=0.
\end{align*}
Since $\curl^2\mathcal H_{\ell m}=0$ and $\Delta=\nabla\operatorname{div}-\curl^2$, the fields $\mathcal H_{\ell m}$ are harmonic
$$
\Delta \mathcal H_{\ell m} = 0,
$$
although their curl is non-zero.
\end{remark*}

\vskip 0.5cm
\begin{theorem}
Let $\bff \in L_p(B_{r_0})$ with $p>1$, satisfying $\divv \bff = 0$, $(\bff,\bn)\big|_{|\bx|=r_0} = 0$, and the orthogonality condition (\ref{ortcondsphere}) for $\ell=1,\ldots,\infty$, $m\in\Z \cap [-\ell,\ell]$. Then there exists a unique solution of the divergence-curl problem (\ref{mainsyst}), (\ref{boundcond}) with boundary conditions (\ref{boundcondsphere}), (\ref{boundcondinfsphere}), given explicitly by the Biot--Savart-Laplace formula
\begin{align} \label{BiotSavar3d}
\bv(\bx) = -\frac{1}{4\pi} \iiint_{B_{r_0}} \frac{(\bx-\by)\times \bff(\by)}{|\bx-\by|^3} \,d\mathbf{y},
\end{align}
with estimates
\begin{align}
&\|\bv(\cdot)\|_{L_q} \leq C   \| \bff \|_{L_p},\quad\frac 1q = \frac 1p - \frac 13,\quad 1<p<q<\infty, \label{est1} \\
&\|\nabla \bv(\cdot)\|_{L_p} \leq C \| \bff \|_{L_p},\quad p>1. \label{est2}
\end{align}
\end{theorem}

\begin{proof}
The solution is given by the explicit formulas (\ref{sol_sphere_1}), (\ref{sol_sphere_2}). The boundary conditions (\ref{boundcondsphere}), (\ref{boundcondinfsphere}) are satisfied by construction. 
We extend the solution $\bv(\by)$ by zero inside the sphere. Then formula (\ref{BiotSavar3d}) becomes the Biot--Savart--Laplace law, which gives the explicit solution of the Cauchy problem for system (\ref{mainsyst}), (\ref{boundcondinfsphere}) (see \cite{GallayWayne}). We now prove that the solution extended by zero is also a solution of the exterior problem. In general, a solution of an elliptic system extended by zero in this way need not be a solution of the Cauchy problem. In spherical harmonics, the system (\ref{mainsyst}) reduces to a system of one-dimensional equations (\ref{mainsyst_sphere_eq1})--(\ref{mainsyst_sphere_eq4}) in the class of functions that are absolutely continuous on $(r_0,\infty)$ with derivative in $L_p$ for some $p>1$. Since the solution of this system vanishes at $r=r_0$ when the conditions of the theorem are satisfied, the solution extended by zero to $[0,\infty)$
 is also absolutely continuous, with derivative remaining in the same class $L_p$. The solution of the Cauchy problem (\ref{mainsyst}), (\ref{boundcondinfsphere}) is unique and is given by formula (\ref{BiotSavar3d}).


The integral in (\ref{BiotSavar3d}) is a Riesz potential, so the first estimate (\ref{est1}) follows from the Hardy--Littlewood--Sobolev inequality (see \cite{Stein}, Chapter~V, \S~1). The function $\nabla\frac{\bx}{|\bx|^3}$ is the kernel of a Calder\'{o}n--Zygmund singular integral operator, so the second estimate follows from the corresponding $L_p$ boundedness result (see Theorem~II.3 in \cite{Calderon}).
\end{proof}

\begin{remark*}
Unlike in a bounded domain, in an unbounded domain the orthogonality condition does not include orthogonality to constants, i.e., the requirement
$$\iiint_{\Omega}\bff \,d\mathbf{x}=0$$
is not necessary. 

In 2D domains for flows with non-zero vorticity and the no-slip condition (\ref{boundcond}), Stokes' theorem yields a non-trivial circulation at infinity:
\begin{align*}
\lim_{R\to\infty}\oint_{\vert\bx\vert=R} \bv \cdot d\mathbf{l} = \iint_{\Omega}\bff \,d\mathbf{x} \neq 0.
\end{align*}
The $L_2$-norm of such solutions is infinite (see \cite{G1}). In the 3D case 
\begin{align*}
\lim_{R\to\infty}\oiint_{\vert\bx\vert=R} \bv \times \bn \, dS = \iiint_{\Omega}\bff \,d\mathbf{x} \neq 0.
\end{align*}
A non-zero integral of $\bff$  also produces circulatory flows at infinity while preserving (\ref{boundcondinfsphere}). Nevertheless, by (\ref{est1}) $\bv$ may have finite $L_2$-norm if $\bff \in L_{6/5}$.
\end{remark*}

\vskip 0.5cm
Replacing condition (\ref{boundcondinfsphere}) with a prescribed uniform flow at infinity
\begin{align}
\bv \to \bv_\infty,~|\bx|\to \infty, \label{boundcondinfsphere2}
\end{align}
the solvability condition instead of (\ref{ortcondspherepolar}) becomes:
\begin{align*} 
\int_{r_0}^\infty
s^{1-\ell}  f_{\ell m}^{(2)}(s)ds = - \frac{r_0^{1-\ell}}{\ell (\ell +1)} v_{\infty, \ell m}^{r},
\end{align*}
or
\begin{align*}
\iiint_{B_{r_0}} \frac{\bff(\bx) \cdot \overline{\mathbf {\Phi}_{\ell m}}\left (\bx/|\bx| \right )}
{|\bx|^{1 + \ell}} \dx = - \frac{r_0^{1-\ell}}{\ell (\ell +1)} v_{\infty, \ell m}^{r},
\end{align*}
where $v_{\infty, \ell m}^{r}$ is the radial component of $\bv_\infty$ in the $\mathbf{Y}_{\ell m}$ direction. Since $v_{\infty, \ell m}^{r}$ vanishes except for isolated $\ell$, $m$, the orthogonality conditions largely remain intact.

The radial and two tangential components $\hat{\mathbf{r}}$, $\hat{\boldsymbol{\theta}}$, $\hat{\boldsymbol{\varphi}}$ of the vector field are related to Cartesian components $\{ \hat {e_k}\}$ by:
\begin{align*}
&\hat {\mathbf  r}=\sin \theta \cos \varphi {\boldsymbol {\hat {e_1 }}}+\sin \theta \sin \varphi {\boldsymbol {\hat {e_2 }}}+\cos \theta {\boldsymbol {\hat {e_3}}},\\
&\hat {\mathbf \theta }=\cos \theta \cos \varphi {\boldsymbol {\hat {e_1 }}}+\cos \theta \sin \varphi {\boldsymbol {\hat {e_2 }}}-\sin \theta {\boldsymbol {\hat {e_3}}},\\
&\hat {\mathbf \varphi }=-\sin \varphi {\boldsymbol {\hat {e_1 }}}+\cos \varphi {\boldsymbol {\hat {e_2 }}}.
\end{align*}

The vector $\bv_{\infty}$ with Cartesian coordinates $(0,0,1)$ has the spherical harmonic expansion:
$$
\bv_{\infty}=\cos \theta \,{\hat {\mathbf {r} }}-\sin \theta \,{\hat {\mathbf {\theta } }}=\mathbf {Y} _{10}+\mathbf{\Psi }_{10}.
$$
Hence $v_{\infty, \ell m}^{r}$ is non-zero only for $\ell=1$, $m=0$.

The solution of \eqref{mainsyst}, \eqref{boundcond}, \eqref{boundcondsphere}, \eqref{boundcondinfsphere2} is given by the formula
$$
\bv(\bx) = -\frac{1}{4\pi} \iiint_{B_{r_0}} \frac{(\bx-\by)\times \bff(\by)}{|\bx-\by|^3} \,d\by + \bv_\infty
$$
and satisfies the estimate
$$
\|\bv - \bv_\infty\|_{L_2(\Omega)} \leq C \|\bff\|_{L_{6/5}(\Omega)},
$$
where $C$ is an absolute constant (by the Hardy-Littlewood-Sobolev inequality).

\section {Conclusion}
In the three-dimensional case, the solvability of the divergence-curl problem with a homogeneous boundary condition is equivalent to orthogonality of the right-hand side to the vector fields from the kernel of $\curl^2$. In the two-dimensional case, the kernel of $\Delta=\curl\nabla^\perp$ plays the same role. For certain spherical domains, it suffices to take a countable family of vector spherical harmonics from the kernel  of $\curl^2$. These fields are solenoidal (but not irrotational) and therefore harmonic. In the exterior of a sphere, when the orthogonality condition is satisfied, the solution is given explicitly by the Biot--Savart--Laplace formula. 

\end{document}